\documentclass{amsart}
\usepackage{amssymb,latexsym,graphics,enumerate,color}
\usepackage{fullpage}
\numberwithin{equation}{section}
\usepackage{bbm}
\usepackage[normalem]{ulem}

\usepackage{xcolor,cancel}

 \usepackage[colorlinks=true]{hyperref}
 \hypersetup{urlcolor=blue, citecolor=blue}

\newtheorem{theorem}{Theorem}[section]
\newtheorem{corollary}[theorem]{Corollary}
\newtheorem{lemma}[theorem]{Lemma}
\newtheorem{proposition}[theorem]{Proposition}

\newcommand{\la}{\langle}
\newcommand{\ra}{\rangle}

\renewcommand{\O}{{\mathcal{O}}}

\newcommand{\R}{{\mathbb{R}}}

\newcommand{\cd}{{\,\cdot\,}}

\newcommand{\ang}{{\not\negmedspace\nabla}}

\newcommand{\C}{{\mathcal{C}}}

\renewcommand{\L}{{\mathcal{L}}}

\newcommand{\tC}{{\tilde{C}}}

\renewcommand{\S}{{\mathbb{S}}}
\newcommand{\supp}{{\text{supp }}}

\begin{document}
\bibliographystyle{plain}

\title[Strauss conjecture with critical damping]
{The Strauss conjecture for the wave equation with critical damping}

\author{Elliot Adamson}
\address{Department of Mathematics, University of North Carolina,
  Chapel Hill, NC  27599-3250, USA}
\email{elliotadamsonunc@outlook.com}

\author{Jason Metcalfe}
\address{Department of Mathematics, University of North Carolina,
  Chapel Hill, NC  27599-3250, USA}
\email{metcalfe@email.unc.edu}

\author{Chengbo Wang}
\address{
  School of Mathematical Sciences\\
  Zhejiang Normal University\\
  Jinhua 321004, China
}
\email{wangcbo@zjnu.edu.cn}

\thanks{The first two authors gratefully acknowledge partial support
  from NSF grant 2135998, and the third author acknowledges the
  support of NSFC 12671293}

\begin{abstract}
The Strauss conjecture concerns wave equations
with power-type nonlinearities and seeks the optimal range of powers
that admit global solutions for all sufficiently nice and sufficiently
small initial data.  In three spatial dimensions, it is well known
that small data global existence is available provided that the power
$p>p_c(3)=1+\sqrt{2}$, where $p_c(n)$ denotes the critical power in
$n$-spatial dimensions.  In the current article, the analogous problem
is considered for damped wave equations where the damping coefficient
is at the critical scaling.  When $a$ denotes the constant multiple
applied to the damping term, it is shown that the corresponding
critical power is $p_c(3+a)$ provided that $a$ is sufficiently small.
\end{abstract}

\maketitle

\section{Introduction}
This article focuses on the analog of the Strauss conjecture for the
wave equation with damping at the critical scaling.  For the free wave
equation, one considers equations modeled by
\[\Box u = |u|^p,\quad (t,x)\in \R_+\times \R^n\]
where $\Box = \partial_t^2-\Delta$.  For sufficiently nice and
sufficiently small data, it is known that global existence holds if
and only if $p>p_c(n)$ where $p_c(n)>1$ solves
\[(n-1)p_c^2-(n+1)p_c-2=0.\]
Notably, $p_c(3)=1+\sqrt{2}$.  When $n=3$, this statement was proved
in \cite{John79}.  In arbitrary dimension, existence was established
in \cite{GLS97}, \cite{Ta01-2}, while blow up for $p<p_c(n)$ was
addressed in \cite{Sideris}.

In the current article, we instead consider
\begin{equation}
  \label{main_equation}
  \begin{cases}
    \partial_t^2w - \Delta w + \frac{a}{1+|x|}\partial_t w =
    F_p(w),\quad (t,x)\in \R_+\times \R^3,\\
    w(0,\cd)=\varepsilon w_{(0)},\quad \partial_t w(0,\cd)=\varepsilon
    w_{(1)}.
  \end{cases}
\end{equation}
We assume that $a>0$ and
\begin{equation}\label{Fp}
  \sum_{0\le j\le 2} |w|^j |\partial_w^j F_p(w)|\lesssim
  |w|^p,\quad\text{ for } |w|\ll 1.
  \end{equation}
  We will work with $w_{(0)}, w_{(1)}\in C^\infty_c(\R^3)$, though it
  is quite likely that this assumption can be relaxed.  While assuming
  that the data are compactly supported, we may without loss of
  generality take $\supp\, w_{(0)},
w_{(1)}\subseteq \{|x|\le \frac{1}{2}\}$.

For an introduction to this problem when the damping term is replaced
by $\frac{a}{(1+|x|)^\beta}$ see \cite{Ikehata-Todorova-Yordanov},
\cite{Lai-Tu}, \cite{Lai-Liu-Tu-Wang}, and
\cite{Nishihara-Sobajima-Wakasugi}.  In short, when $\beta<1$, the
parabolic behavior dominates, and the critical exponent is related to
the Fujita exponent.  On
the other hand, when $\beta>1$, the damping is too weak to alter the
Strauss exponent.  See \cite{AMW} for a proof of small data global
existence for $p>1+\sqrt{2}$ and \cite{Lai-Liu-Tu-Wang} for some upper
bounds on the lifespan when $p\le 1+\sqrt{2}$.

The current article focuses on the critical $\beta = 1$ case.  The
conjectured critical exponent is $p_c(3+a)$ for $0<a\le 1$.
The intuition centers around the radial part of the linear operator
near spatial infinity, which looks like
\[\partial_t^2 - \partial_r^2 - \frac{2}{r}\partial_r +
  \frac{a}{r}\partial_t,\]
and can be rewritten as
\[\partial_t^2 - \partial_r^2 - \frac{2+a}{r}\partial_r +
  \frac{a}{r}(\partial_t+\partial_r).\]
The latter resembles a wave equation in ``$(3+a)$-dimensions'' along with a
first order perturbation that has the structure of a ``good''
derivative and is anticipated to be harmless.

When $p\le p_c(3+a)$, see
\cite{Li} for some upper bounds on the lifespan when $a$ is large
enough that the Fujita exponent comes into play and
\cite{Lai-Liu-Tu-Wang} for smaller $a$.  The main result of this paper
gives global existence for $p>p_c(3+a)$ for $a$ sufficiently small.
As an artifact of the techniques, we require $0<a\le\frac{1}{7}$, which
is not the full range that is conjectured.  I.e.,
we require that $p>\max(p_c(3+a), \frac{7}{3})$.  For global existence
when $a> 1$, where the parabolic behavior enters, see \cite{Sobajima}.

In order to prove global existence, we will study a conjugated
equation.  It is tempting to conjugate by $(1+|x|)^{\frac{a}{2}}$, and
in essence, this is what we do.  Using this, however, would result in
terms in the commutator of the linear operator and the translational
vector fields that are singular near $r=0$.  To avoid this, we fix
$\beta\in C^\infty(\R)$ nondecreasing with
\[
  \beta(r)=\begin{cases}1,&r\ge 2,\\
    0,&r\le 1.\end{cases}
\]
Defining
\[\phi(r)=\int_0^r \frac{\beta(\rho)}{1+\rho}\,d\rho,\]
we set
\begin{equation}\label{v}v = e^{\frac{a}{2}\phi(r)}w.\end{equation}
One can readily check that there exist $c, C>0$ so that
\begin{equation}\label{vw}c(1+r)^{\frac{a}{2}} w \le v \le C (1+r)^{\frac{a}{2}} w.\end{equation}
We may then compute
\begin{equation}
  \label{conjugated}
  \begin{cases}
\Box v +
\frac{a\beta}{1+r}\Bigl(\partial_t+\partial_r+\frac{1}{r}\Bigr) v -
\frac{a(a+2)\beta}{4(1+r)^2} v
+ \frac{a(1-\beta)}{1+r}\partial_tv + \frac{a\beta'}{2(1+r)} v +
\frac{a^2\beta (1-\beta)}{4(1+r)^2} v= e^{\frac{a}{2}\phi}
F_p(e^{-\frac{a}{2}\phi} v),\\
v(0,\cd)=\varepsilon v_{(0)}, \quad \partial_t v(0,\cd)=\varepsilon v_{(1)}.
\end{cases}
\end{equation}
We will use $\L$ to denote the linear operator
\[
\L = \Box  +
\frac{a\beta(r)}{1+r}\Bigl(\partial_t+\partial_r+\frac{1}{r}\Bigr) -
\frac{a(a+2)\beta(r)}{4(1+r)^2} 
+ \frac{a(1-\beta(r))}{1+r}\partial_t + \frac{a\beta'(r)}{2(1+r)}  +
\frac{a^2\beta(r) (1-\beta(r))}{4(1+r)^2}.
  \]
In particular, we note that the last three terms in the left side of \eqref{conjugated} are
supported in $\{r\le 2\}$.
If we show that $v$ exists globally for $\varepsilon$ sufficiently
small and smooth and
compactly supported $v_{(0)}, v_{(1)}$, then the corresponding
existence statement for $w$ is immediate.

The main result then states:
\begin{theorem}\label{main_thm}
Let $a>0$ and $\max(p_c(3+a), \frac{7}{3}) < p < \frac{4}{2+a}+1$, and suppose that $v_{(0)},
v_{(1)}\in C^\infty_c(\R^3)$.  Then if $\varepsilon>0$ is sufficiently
small, there exists a global solution $v$ to \eqref{conjugated}.
\end{theorem}
For such small data existence results, the most interesting behavior
occurs for small values of $p$.  For powers exceeding
$\frac{4}{2+a}+1$, there are standard modifications that will yield
global existence once it is understood for the given range.

As an immediate consequence, we see that
\begin{corollary}
For $0<a\le\frac{1}{7}$ and $p>p_c(3+a)$, if $w_{(0)}, w_{(1)}\in
C^\infty_c$ and $\varepsilon>0$ is sufficiently small, then
\eqref{main_equation} admits a global solution $w$.
\end{corollary}

We note that there are also significant works in the case that the
damping is of the form $\frac{a}{(1+t)^\beta}\partial_t$.  See
e.g. \cite{He-Lai}, \cite{Lai-Takamura}, \cite{Lai-Zhou},
\cite{Dabbicco-Lucente-Reissig}, \cite{Li-Yin}, \cite{Liu-Wang}.  The work most
directly analogous to the current study is \cite{Dabbicco-Lucente-Reissig}.

The proof is related to that of \cite{AMW}.  Rather than relying on
weighted Strichartz estimates such as in the recent studies
\cite{HMSSZ}, \cite{FaWa}, \cite{LMSTW}, and \cite{MW}, we develop a
method that is based solely in (weighted) integrated local energy
estimates.  We employ both the $r^p$-weighting of \cite{DaRo_rp} and
the ghost weighting of \cite{Alinhac_ghost}.  See \cite{M-Rhoads} and
\cite{M-Stewart}.  These bounds effectively handle the ``good derivatives'' that
result upon conjugating, and after applying
space-time Hardy inequalities they give better bounds on the
undifferentiated solution than the standard integrated local energy
bounds.  Unlike \cite{AMW}, one must be careful with the constants in
this proof.  The zeroth order terms in the left side of
\eqref{conjugated} are not necessarily perturbative, even outside of a
compact set.
This is where the $a<\frac{1}{7}$ restriction arises.  To handle the
nonlinear terms in the resulting $L^2_{t,x}$ norm, variants of the space-time
Klainerman Sobolev estimates of \cite{MTT} are called upon.  We, in
particular, use the approach of \cite{AMW} that proves such while
critically, due to the limited regularity of $F_p$, only costs two
vector fields.

\vskip .125in
\subsection{Notation}
We will use the orthogonal decomposition
\[\nabla = \omega\partial_r + \ang,\qquad
  \partial_r=\omega\cdot\nabla.\]
Here and throughout, we denote $r=|x|$ and $\omega = \frac{x}{r}$.  We use $\la r\ra = 1 +r$
(and similarly $\la t-r\ra = 1+t-r$ when $1+t-r>c$ for some $c>0$).
We let $\partial = (\partial_t,\nabla)$ denote the space-time
gradient and reserve $\nabla$ for the spatial gradient.
We will use
\[S=t\partial_t+r\partial_r,\qquad \Omega=x\times\nabla,\qquad Z=(\nabla,\Omega,S)\]
to denote the admissible vector fields on which our proof relies.  By
avoiding the use of the Lorentz boosts, there is a likelihood that the
proof can be adapted to exterior domains or certain asymptotically
flat background geometries.  When vector fields are applied to a
function, we shorten the notation with
\[|AZ^{\le N} w| = \sum_{|\mu|\le N} |AZ^\mu w|,\qquad |A\nabla^{\le N}
  w| = \sum_{|\mu|\le N} |A\nabla^\mu w|,\qquad |A\partial^{\le N} w| =
  \sum_{\mu\le N} |A\partial^\mu w|\]
when $A$ is a vector field.

The decay estimates that we use to close the global existence argument
are variants of those derived in \cite{MTT}.  See also the expanded
presentations in, e.g., \cite{M-Rhoads}, \cite{M-Stewart}, and
\cite{Looi-Tohaneanu}.  And see \cite{AMW} for the formulation used herein.  The exact form of the
bounds depends on the proximity to $r=0$ and to the light cone $t=r$.
Because we are assuming that the Cauchy data are supported in $\{r\le
\frac{1}{2}\}$, the solution $v$ will vanish outside of $C=\{r\le
t+1\}$.  We first dyadically subdivide this region in $t$:
\[C_\tau = C\cap \{\la t\ra\in [\tau, 2\tau]\},\quad C=\bigcup_\tau C_\tau.\]
We next set
\[C_\tau^{R=1}=C_\tau\cap \{r\le 2\},\quad C^R_\tau = C_\tau \cap
  \{r\in [R, 2R]\}\text{ for } R>1,\]
\[C^{U=1}_\tau = C_\tau\cap \{|t-r|\le 2\},\quad C^U_\tau = C_\tau
  \cap \{t-r\in [U, 2U]\}\text{ for } U>1.\]
Including
\[C^{\tau/2}_\tau = C_\tau \cap \{r\ge \tau/2\}\cap \{t-r\ge
  \tau/2\},\]
we have
\[C_\tau = \Bigl(\bigcup_{R\le \tau/4} C^R_\tau\Bigr) \cup \Bigl(\bigcup_{U\le \tau/4}
  C^U_\tau\Bigr)\cup C^{\tau/2}_\tau.\]
The parameters $R, U, \tau$ should always be understood to vary over
dyadic values.
Importantly,
\[\text{on } C^R_\tau:\quad \la r\ra\approx R,\quad t-r\approx t \approx
  \tau,\qquad\text{ and }\qquad\text{ on } C^U_\tau:\quad \la t-r\ra\approx U,\quad r\approx t\approx
  \tau.\]
We also have
\[\text{on } C^{\tau/2}_\tau:\quad t-r\approx r\approx t\approx
  \tau,\]
and we will permit the abuse of notation $C^{\tau/2}_\tau =
C^{R=\tau/2}_\tau$.  We will use $\tC^R_\tau$, $\tC^U_\tau$ to denote slight enlargements
in both parameters, which will be used to capture the tails of cutoffs
that are introduced before applying Sobolev embeddings.

We introduce the notation
\[\|v\|_{\ell^q_\tau \ell^p_{U\le \tau/4} L^2_tL^2_x(C^U_\tau)}
    = \Bigl[\sum_\tau \Bigl(\sum_{1\le U\le \tau/4} \|v\|_{L^2_tL^2_x(C^U_\tau)}^p\Bigr)^{q/p}\Bigr]^{1/q}.\]
We make the obvious modification for when $p$ or $q$ is infinite.  We
may also define the analogous sum for $C^R_\tau$.

\vskip .25in
\section{Local energy estimates}
The principal estimate under consideration in this section is an
integrated local energy estimate that combines the $r^p$-weighting
ideas of \cite{DaRo_rp} and ghost weighting of \cite{Alinhac_ghost}.
See, e.g., \cite{M-Rhoads}, \cite{M-Stewart}.

In order to handle some compactly supported error terms, it is useful
to first record the local energy-type bounds that result from the
damping.  These are more readily proved prior to conjugating.  As
such, we assume that $w$ solves
\begin{equation}
  \label{boxw}
  \begin{cases}
    \Box w + \frac{a}{1+r}\partial_t w = F,\\
    w(0,\cd)=w_{(0)}, \quad \partial_t w(0,\cd)=w_{(1)}
  \end{cases}
\end{equation}
with
\begin{equation}
  \label{suppw}
  \supp\, w_{(0)}, w_{(1)}\subseteq \{|x|\le \frac{1}{2}\}, \quad
  \supp\, F\subseteq \{|x|\le t+\frac{1}{2}\}.
\end{equation}

We first record the standard energy inequality
\begin{equation}\label{energy}\frac{1}{2} \|\partial w(T,\cd)\|_{ L^2}^2 + a\|\la r\ra^{-\frac{1}{2}}
  \partial_t w\|^2_{L^2_{t,x}([0,T]\times\R^3)}
  \le \frac{1}{2} \|\partial w(0,\cd)\|_{L^2}^2 + \int_0^T \int_{\R^3}
  \Bigl|\Bigl(\Box + \frac{a}{1+r}\partial_t\Bigr)w\Bigr||\partial_t
  w|\,dx\,dt,\end{equation}
which holds for any $T$.
We shall use this to derive integrated local energy-type bounds on
$w$.  We apply the lower order correction that appeared in
\cite{Sterbenz}, \cite{MetSo06} to cancel an occurrence of the
Lagrangian in the derivation of integrated local energy estimates.
\begin{lemma}
  Suppose that $w$ solves \eqref{boxw} with $a>0$ and \eqref{suppw}.  Then
  \begin{multline}
    \label{energyw}
    \|r^{-\frac{1}{2}}\la r\ra^{-1}w\|_{\ell^\infty_R \ell^2_\tau L^2_{t,x}(C^R_\tau)} + \|\la
    r\ra^{-\frac{1}{2}}\partial w\|_{\ell^\infty_R\ell^2_\tau L^2_{t,x}(C^R_\tau)}
    \\\lesssim \|\partial w(0,\cd)\|_{L^2} + \|\la
    r\ra^{-1}w(0,\cd)\|_{L^2} + \Bigl\|\la r\ra^{\frac{1}{2}}\Bigl(\Box + \frac{a}{\la
      r\ra}\partial_t\Bigr) w\Bigr\|_{\ell^1_R\ell^2_\tau
      L^2_{t,x}(C^R_\tau)}.
  \end{multline}
\end{lemma}

\begin{proof}
By fixing $R\ge 1$ and applying integration by parts to
  \[\int_0^T\int_{\R^3} \Bigl(\Box + \frac{a}{\la
      r\ra}\partial_t\Bigr) w \cdot \frac{1}{r+R} w\,dx\,dt,\]
  we see that it is equivalent to
  \begin{multline*}
    \int_{\R^3} \frac{1}{r+R} w\partial_t w\,dx\Bigl|_0^T
+\frac{a}{2} \int_{\R^3} \frac{1}{(r+1)(r+R)} w^2\,dx\Bigl|_0^T
    -\int_0^T\int_{\R^3}\frac{1}{r+R}(\partial_tw)^2\,dx\,dt
    \\+ \int_0^T\int_{\R^3} \frac{1}{r+R} |\nabla w|^2\,dx\,dt
    +\int_0^T\int_{\R^3} \frac{R}{r(r+R)^3} w^2\,dx\,dt.
  \end{multline*}
  Thus, by combining with \eqref{energy} and dropping nonnegative
  terms on the $t=T$ slice, we have
  \begin{multline*}
\|r^{-\frac{1}{2}}\la r\ra^{-1}w\|^2_{L^2_{t,x}([0,T]\times\{\la
  r\ra\approx R\})} + \|\la
    r\ra^{-\frac{1}{2}}\partial w\|^2_{L^2_{t,x}([0,T]\times \{\la
      r\ra\approx R\})}
    \\\lesssim \|\partial w(0,\cd)\|^2_{L^2} + \|\la
    r\ra^{-1}w(0,\cd)\|^2_{L^2} + \int_0^T\int_{\R^3} \Bigl|\Bigl(\Box +
    \frac{a}{\la r\ra}\partial_t\Bigr)w\Bigr| \Bigl(|\partial_t w| +
    \la r\ra^{-1} |w|\Bigr)\,dx\,dt.
  \end{multline*}
We may now take supremums in $T, R$.  We also apply the Schwarz
inequality  to bound
\begin{multline*}\int_0^\infty \int_{\R^3}\Bigl|\Bigl(\Box +
    \frac{a}{\la r\ra}\partial_t\Bigr)w\Bigr| \Bigl(|\partial_t w| +
    \la r\ra^{-1} |w|\Bigr)\,dx\,dt
  \\ \lesssim
    \Bigl\|\la r\ra^{\frac{1}{2}}\Bigl(\Box + \frac{a}{\la
      r\ra}\partial_t\Bigr) w\Bigr\|_{\ell^1_R\ell^2_\tau
      L^2_{t,x}(C^R_\tau)}
    \Bigl(\|\la r\ra^{-\frac{1}{2}}\partial
    w\|_{\ell^\infty_R\ell^2_\tau L^2_{t,x}(C^R_\tau)} + \|\la
    r\ra^{-\frac{3}{2}}w\|_{\ell^\infty_R\ell^2_\tau L^2_{t,x}(C^R_\tau)}\Bigr).
  \end{multline*}
  The last factor can be absorbed back into the left, which gives \eqref{energyw}.
\end{proof}

For $(t,x)\in \R_+\times\R^3$, we let $v=e^{\frac{a}{2}\phi(r)}w$, which solves
\begin{equation}
  \label{boxv}
  \begin{cases}
  \L v =e^{\frac{a}{2}\phi}F,\\
  v(0,\cd)=v_{(0)},\quad \partial_t v(0,\cd)=v_{(1)}
\end{cases}
\end{equation}
where we assume
\begin{equation}
  \label{supp}
  \supp\, v_{(0)}, v_{(1)}\subseteq \{|x|\le \frac{1}{2}\}, \quad
  \supp\, F\subseteq \{|x|\le t+\frac{1}{2}\}.
\end{equation}

  Using \eqref{v}, \eqref{vw}, and the fact that
  \[\Bigl(\Box + \frac{a}{\la r\ra} \partial_t\Bigr) w =
    e^{-\frac{a}{2}\phi(r)} \L v,\]
  we may translate \eqref{energyw} into a statement for $v$.
\begin{lemma}  Let $v\in C^2(\R_+\times\R^3)$ solve \eqref{boxv} with
  $a>0$ and \eqref{supp}.  Then, 
  \begin{equation}
      \label{energyv}
    \|r^{-\frac{1}{2}}\la r\ra^{-(1+\frac{a}{2})}v\|_{\ell^\infty_R \ell^2_\tau L^2_{t,x}(C^R_\tau)} + \|\la
    r\ra^{-\frac{1+a}{2}}\partial v\|_{\ell^\infty_R\ell^2_\tau L^2_{t,x}(C^R_\tau)}
    \lesssim \|\partial v(0,\cd)\|_{L^2}\\
    + \Bigl\|\la r\ra^{\frac{1-a}{2}}\L v\Bigr\|_{\ell^1_R\ell^2_\tau
      L^2_{t,x}(C^R_\tau)}.
  \end{equation}
\end{lemma}
Here we have applied a Hardy inequality
    \begin{equation}\label{L2Hardy}
      \|r^{-1} f\|_{L^2(\R^3)} \lesssim \|\nabla f\|_{L^2(\R^3)}
      \end{equation}
and \eqref{supp} to
simplify the initial data terms.

We are also going to need higher order versions of
\eqref{energyv}.  We first calculate the appropriate
commutators:
\begin{equation}
  \label{Scomm}
  [\L,S]=2\L - \frac{a\beta(r)}{\la r\ra^2}
  \Bigl(\partial_t+\partial_r+\frac{1}{r}\Bigr) +
  \frac{a(a+2)\beta(r)}{2\la r\ra^3}
  + V_{S,0}(r) \partial_t + V_{S,1}(r)\partial_r + V_{S}(r),
\end{equation}
\begin{equation}
  \label{Omegacomm}
  [\L,\Omega]=0,
\end{equation}
\begin{multline}
  \label{dcomm}
  [\L,\nabla_{x_j}] = \omega_j \frac{a\beta(r)}{\la r\ra^2}
  \Bigl(\partial_t+\partial_r+\frac{1}{r}\Bigr)
  - \frac{a\beta(r)}{r\la r\ra} \ang_j +
  \omega_j\frac{a\beta(r)}{r^2\la r\ra}
  -\omega_j\frac{a(a+2)\beta(r)}{2\la r\ra^3}
  \\+ \omega_j V_{\nabla, 0}(r) \partial_t + \omega_j V_{\nabla,1}(r)
  \partial_r + \omega_j V_{\nabla}(r).
\end{multline}
Here $V_{S}, V_{S,0}, V_{S,1}, V_\nabla, V_{\nabla,0}, V_{\nabla,1}$
are smooth functions of $r$.  $V_S, V_\nabla, V_{S,1},$ and $V_{\nabla,1}$ are
  supported in $\{1\le r\le 2\}$, while $V_{S,0}$ and $V_{\nabla,0}$ are
  supported within $\{r\le 2\}$.  The former group may thus be
  regarded as functions in $C^\infty(\R^3)$ while the latter are only
  $C^1(\R^3)$.  Both suffice for our purposes below.
  We must, however, dedicate extra care to the
lack of differentiability of $\omega_j$ at $r=0$ for the $V_{\nabla,
  0}$ term.

\begin{corollary}\label{cor_energyZv}
 Let $v\in C^4(\R_+\times\R^3)$ solve \eqref{boxv} with $a>0$ and
 \eqref{supp}.  Then, for $s<\frac{1}{2}$, we have
 \begin{multline}
   \label{energyZv}
    \|r^{-\frac{1}{2}}\la r\ra^{-(1+\frac{a}{2})} Z^{\le 2} v \|_{\ell^\infty_R
      \ell^2_\tau L^2_{t,x}(C^R_\tau)}
    + \|\la r\ra^{-\frac{1+a}{2}} \partial Z^{\le 2}
    v\|_{\ell^\infty_R\ell^2_\tau L^2_{t,x}(C^R_\tau)}
    \lesssim \|\partial \nabla^{\le 2} v(0,\cd)\|_{L^2} \\+\inf_{F=F_1+F_2}\Bigl(
\|\la r\ra^{\frac{3}{2}-s}Z^{\le 2}F_1\|_{L^2_{t,x}}
    + \|\la r\ra^{1-s}\la
    t-r\ra^{\frac{1}{2}}Z^{\le 2}F_2\|_{\ell^1_U\ell^2_{\tau} L^2_{t,x}(C^U_\tau)}\Bigr)
    .
 \end{multline}
\end{corollary}

\begin{proof}
  Applying \eqref{energyv} to $Z^{\le 1}v$ and using \eqref{Scomm},
  \eqref{Omegacomm}, and \eqref{dcomm} gives
\begin{multline*}
    \|r^{-\frac{1}{2}}\la r\ra^{-(1+\frac{a}{2})} Z^{\le 1} v \|_{\ell^\infty_R
      \ell^2_\tau L^2_{t,x}(C^R_\tau)}
    + \|\la r\ra^{-\frac{1+a}{2}} \partial Z^{\le 1}
    v\|_{\ell^\infty_R\ell^2_\tau L^2_{t,x}(C^R_\tau)}
    \lesssim \|\partial \nabla^{\le 1} v(0,\cd)\|_{L^2} \\+ \|\la
    r\ra^{\frac{1-a}{2}} Z^{\le 1} \L v\|_{\ell^1_R \ell^2_\tau
      L^2_{t,x}(C^R_\tau)}
    +\|\la r\ra^{-\frac{3+a}{2}} \partial v\|_{\ell^1_R
        \ell^2_\tau L^2_{t,x}(C^R_\tau)}
      + \|\la r\ra^{-\frac{5+a}{2}} 
      v\|_{\ell^1_R\ell^2_\tau L^2_{t,x}(C^R_\tau)}.
    \end{multline*}
        Since
    \begin{multline*}
    \|\la r\ra^{-\frac{3+a}{2}} \partial v\|_{\ell^1_R
        \ell^2_\tau L^2_{t,x}(C^R_\tau)}
      + \|\la r\ra^{-\frac{5+a}{2}} 
      v\|_{\ell^1_R\ell^2_\tau L^2_{t,x}(C^R_\tau)}
\\\lesssim \|\la r\ra^{-\frac{1+a}{2}} \partial v\|_{\ell^\infty_R
        \ell^2_\tau L^2_{t,x}(C^R_\tau)}
      + \|r^{-\frac{1}{2}}\la r\ra^{-(1+\frac{a}{2})}
      v\|_{\ell^\infty_R\ell^2_\tau L^2_{t,x}(C^R_\tau)},
    \end{multline*}
    we may subsequently apply \eqref{energyv} to obtain
\begin{multline}\label{energyZva}
    \|r^{-\frac{1}{2}}\la r\ra^{-(1+\frac{a}{2})} Z^{\le 1} v \|_{\ell^\infty_R
      \ell^2_\tau L^2_{t,x}(C^R_\tau)}
    + \|\la r\ra^{-\frac{1+a}{2}} \partial Z^{\le 1}
    v\|_{\ell^\infty_R\ell^2_\tau L^2_{t,x}(C^R_\tau)}
    \\\lesssim \|\partial \nabla^{\le 1} v(0,\cd)\|_{L^2} + \|\la
    r\ra^{\frac{1-a}{2}} Z^{\le 1} \L v\|_{\ell^1_R \ell^2_\tau
      L^2_{t,x}(C^R_\tau)}.
    \end{multline}

 In turn, we use \eqref{energyZva} on $Z^{\le 1} v$.  Here we must
 take care with the possibility that the vector field within the $\L
 v$ term is a derivative and it lands on the $\omega_j$ within the
 $V_{\nabla,0}$ term of \eqref{dcomm}.  This yields
  \begin{multline*}
    \|r^{-\frac{1}{2}}\la r\ra^{-(1+\frac{a}{2})} Z^{\le 2} v \|_{\ell^\infty_R
      \ell^2_\tau L^2_{t,x}(C^R_\tau)}
    + \|\la r\ra^{-\frac{1+a}{2}} \partial Z^{\le 2}
    v\|_{\ell^\infty_R\ell^2_\tau L^2_{t,x}(C^R_\tau)}
    \lesssim \|\partial \nabla^{\le 2} v(0,\cd)\|_{L^2} \\+ \|\la
    r\ra^{\frac{1-a}{2}} Z^{\le 2} \L v\|_{\ell^1_R \ell^2_\tau
      L^2_{t,x}(C^R_\tau)}
    +\|\la r\ra^{-\frac{3+a}{2}} \partial Z^{\le 1} v\|_{\ell^1_R
        \ell^2_\tau L^2_{t,x}(C^R_\tau)}
   \\   + \|\la r\ra^{-\frac{5+a}{2}} 
      Z^{\le 1}  v\|_{\ell^1_R\ell^2_\tau L^2_{t,x}(C^R_\tau)}
      + \|r^{-1} (1-\beta(r/2)) \partial_t v\|_{L^2_{t,x}}.
    \end{multline*}
   For the last term, we use the Hardy inequality \eqref{L2Hardy}.
We, thus, see that
          \begin{multline*}
    \|\la r\ra^{-\frac{3+a}{2}} \partial Z^{\le 1} v\|_{\ell^1_R
        \ell^2_\tau L^2_{t,x}(C^R_\tau)}
      + \|\la r\ra^{-\frac{5+a}{2}} Z^{\le 1}
      v\|_{\ell^1_R\ell^2_\tau L^2_{t,x}(C^R_\tau)} + \|r^{-1}(1-\beta(r/2))\partial_tv\|_{L^2_{t,x}}
\\\lesssim \|\la r\ra^{-\frac{1+a}{2}} \partial Z^{\le 1}v\|_{\ell^\infty_R
        \ell^2_\tau L^2_{t,x}(C^R_\tau)}
      + \|r^{-\frac{1}{2}}\la r\ra^{-(1+\frac{a}{2})}
      Z^{\le 1} v\|_{\ell^\infty_R\ell^2_\tau L^2_{t,x}(C^R_\tau)},
    \end{multline*}
which allows these error terms to be controlled by \eqref{energyZva}
and gives
\begin{multline}\label{energyZ2va}
    \|r^{-\frac{1}{2}}\la r\ra^{-(1+\frac{a}{2})} Z^{\le 2} v \|_{\ell^\infty_R
      \ell^2_\tau L^2_{t,x}(C^R_\tau)}
    + \|\la r\ra^{-\frac{1+a}{2}} \partial Z^{\le 2}
    v\|_{\ell^\infty_R\ell^2_\tau L^2_{t,x}(C^R_\tau)}
    \\\lesssim \|\partial \nabla^{\le 2} v(0,\cd)\|_{L^2} + \|\la
    r\ra^{\frac{1-a}{2}} Z^{\le 2} \L v\|_{\ell^1_R \ell^2_\tau
      L^2_{t,x}(C^R_\tau)}.
    \end{multline}

We end by modifying the norm of the forcing term to match what will occur in
other estimates later.  For $R\approx \tau$, we break $C^{R\approx \tau}_\tau$ into the
 $U$ regions.  We then notice that
\begin{equation}
  \label{energyForcing}
  \|\la r\ra^{\frac{1}{2}} F\|_{\ell^1_R\ell^2_\tau
    L^2_{t,x}(C^R_\tau)}
  \lesssim \|\la r\ra^{\frac{3}{2}-s}F_1\|_{L^2_{t,x}}
    + \|\la r\ra^{1-s}\la
    t-r\ra^{\frac{1}{2}}F_2\|_{\ell^1_U\ell^2_\tau L^2_{t,x}(C^U_\tau)}
\end{equation}
provided that $F=F_1+F_2$ and that $s<\frac{1}{2}$.  
  \end{proof}

We now turn our attention to the weighted versions of the integrated
local energy estimates.
We start with a space-time Hardy inequality.  This estimate has
previously appeared in, e.g., \cite[Lemma 2.2]{M-Stewart}.  In the
current study, however, the constant will play a role, so we provide
the proof with that added level of care.  We set
$\sigma_\theta(z)=\frac{z}{\theta+|z|}$, which is the function that
will be used in the ghost weights.  

\begin{lemma}
  Let $s>0$.  Suppose $v\in C^2(\R_+\times\R^3)$ and that for every
  $T\in \R_+$ there exists $R>0$ so that $v(t,x)\equiv 0$ on
  $[0,T]\times\{r\ge R\}$.  Then,
  \begin{multline}
    \label{HardyR}
    \|\la r\ra^{-\frac{1}{2}-s} \sqrt{\beta} e^{-\frac{\sigma_U(t-r)}{2}}
    v\|^2_{L^2_{t,x}} + \frac{1}{s}\|\la r\ra^{-s} \sqrt{\beta} e^{-\frac{\sigma_U(t-r)}{2}}
    v\|_{L^\infty_tL^2_x}^2
    \le \frac{1}{s} \|\la r\ra^{-s} \sqrt{\beta} e^{-\frac{\sigma_U(-r)}{2}} v(0,\cd)\|_{L^2}^2   \\+ \frac{1}{s^2}
    \|\la r\ra^{\frac{1}{2}-s}\sqrt{\beta} e^{-\frac{\sigma_U(t-r)}{2}}
    r^{-1}(\partial_t+\partial_r)(rv)\|_{L^2_{t,x}}^2 +
    \frac{1}{s}\|\la r\ra^{-s} \sqrt{\beta'}
    e^{-\frac{\sigma_U(t-r)}{2}} v\|_{L^2_{t,x}}^2 
  \end{multline}
  for any $\beta\in C^1$ that is non-negative and non-decreasing.  In
  particular, it holds for the $\beta$ within \eqref{v} or for
  $\beta\equiv 1$.
\end{lemma}

\begin{proof}
  By writing
  \[\int_0^T\int_{\R^3} \beta(r) (1+r)^{-1-2s} e^{-\sigma_U(t-r)} v^2\,dx\,dt
    = -\frac{1}{2s}\int_0^T\int_{\S^2}\int_0^\infty \beta(r) e^{-\sigma_U(t-r)}
    (\partial_t+\partial_r)(1+r)^{-2s} (rv)^2\,dr\,d\omega\,dt,\]
  we may integrate by parts to see that this is equivalent to
  \begin{multline*}-\frac{1}{2s} \int_{\R^3} \beta(r) (1+r)^{-2s} e^{-\sigma_U(t-r)} v^2\,dx\Bigl|_{t=0}^T
+\frac{1}{2s} \int_{\R^3} \beta'(r) (1+r)^{-2s} e^{-\sigma_U(t-r)} v^2\,dx\,dt
\\    + \frac{1}{s} \int_0^T\int_{\S^2}\int_0^\infty \beta(r)(1+r)^{-2s}e^{-\sigma_U(t-r)}
    (rv)(\partial_t+\partial_r)(rv)\,dr\,d\omega\,dt
    .\end{multline*}
 By the Schwarz inequality and Young's inequality, the last term is
 bounded by
 \[\frac{1}{2}\int_0^T\int_{\R^3} \beta(r)(1+r)^{-1-2s}e^{-\sigma_U(t-r)} v^2\,dx\,dt
   + \frac{1}{2s^2}\int_0^T\int_{\R^3}\beta(r)(1+r)^{1-2s}e^{-\sigma_U(t-r)}
   [r^{-1}(\partial_t+\partial_r)(rv)]^2\,dx\,dt.\]
 Applying this bound and taking a supremum in $T$ gives
 \begin{multline*}\frac{1}{2}\|\la r\ra^{-\frac{1}{2}-s} \sqrt{\beta} e^{-\frac{\sigma_U(t-r)}{2}} v\|_{L^2_{t,x}}^2
   +\frac{1}{2s} \|\la r\ra^{-s} \sqrt{\beta} e^{-\frac{\sigma_U(t-r)}{2}} v\|^2_{L^\infty_tL^2_x}
   \le \frac{1}{2s} \|\la r\ra^{-s}\sqrt{\beta} e^{-\frac{\sigma_U(-r)}{2}} v(0,\cd)\|^2_{L^2} \\+
   \frac{1}{2s^2} \|\la r\ra^{\frac{1}{2}-s}\sqrt{\beta} 
   e^{-\frac{\sigma_U(t-r)}{2}}
   r^{-1}(\partial_t+\partial_r)(rv)\|^2_{L^2_{t,x}} + \frac{1}{2s}
   \|\la r\ra^{-s} \sqrt{\beta'}e^{-\frac{\sigma_U(t-r)}{2}} v\|_{L^2_{t,x}}^2,
 \end{multline*}
 which is equivalent to \eqref{HardyR}. 
\end{proof}

We will need a related Hardy inequality on the $C^U_\tau$ regions.  We
record the following, which was proved in \cite[Lemma 3.7]{AMW}.
\begin{lemma}
  Let $0\le \sigma< \min(s,1/2)$.  Suppose $v\in C^2(\R_+\times\R^3)$
  has support contained in $\{r\le t+\frac{1}{2}\}$.  Then
  \begin{multline}
    \label{HardyU}
    \|\la r\ra^{-\frac{1}{2}-s+\sigma} \la t-r\ra^{-\sigma}
    v\|_{\ell^{\frac{2}{1-2\sigma}}_U \ell^2_\tau L^2_{t,x}(C^U_\tau)} \lesssim
      \|v(0,\cd)\|_{L^2}
      \\+ \|\la r\ra^{\frac{1}{2}-s}
      r^{-1}(\partial_t+\partial_r)(rv)\|^{1-2\sigma}_{L^2_{t,x}}
      \|\la r\ra^{1-s} \la t-r\ra^{-\frac{1}{2}}
      r^{-1}(\partial_t+\partial_r)(rv)\|^{2\sigma}_{\ell^\infty_U
        \ell^2_\tau L^2_{t,x}(C^U_\tau)}.
  \end{multline}
\end{lemma}

Here one writes
\[(1+r)^{-1-2s+2\sigma}=-\frac{1}{2s-2\sigma} (\partial_t+\partial_r)(1+r)^{-2s+2\sigma},\]
integrates by parts, and argues much like in the preceding lemma.  An indicator function that localizes to $\la
t-r\ra\approx U$ is included.  For the resulting term where $v$ is
differentiated, H\"older's inequality is applied, which gives
\begin{multline}
  \label{HolderU}
  \|\la r\ra^{\frac{1}{2}-s+\sigma}\la t-r\ra^{-\sigma}
  r^{-1}(\partial_t+\partial_r)(rv)\|_{\ell_U^{\frac{2}{1-2\sigma}}
    \ell^2_\tau L^2_{t,x}(C^U_\tau)}
  \\\lesssim \|\la r\ra^{\frac{1}{2}-s}
  r^{-1}(\partial_t+\partial_r)(rv)\|^{1-2\sigma}_{L^2_{t,x}}
  \|\la r\ra^{1-s} \la
  t-r\ra^{-\frac{1}{2}}r^{-1}(\partial_t+\partial_r)(rv)\|^{2\sigma}_{\ell^\infty_U
    \ell^2_\tau L^2_{t,x}(C^U_\tau)}.
\end{multline}

We now prove the main integrated local energy estimate for $v$.  This
uses both $r^p$-weights \cite{DaRo_rp} and ghost weights
\cite{Alinhac_ghost}.  While this estimate has appeared previously in
\cite{M-Rhoads} and \cite{M-Stewart}, we reprove it here as it is
necessary to carefully track the constants in order to control the
zeroth order term in \eqref{boxv}.

\begin{proposition} Suppose $a, s$ satisfy
$0<a<2s<1$.  Then for $v\in C^2(\R_+\times\R^3)$ solving
\eqref{boxv} subject to \eqref{supp}, we have
\begin{multline}
  \label{rpd}
 \|\la r\ra^{1-s}
 r^{-1}(\partial_t+\partial_r)(rv)\|_{L^\infty_tL^2_x} + \|\la
 r\ra^{1-s} \ang v\|_{L^\infty_tL^2_x}
 + \|\la r\ra^{\frac{1}{2}-s}
 r^{-1}(\partial_t+\partial_r)(rv)\|_{L^2_{t,x}}
 \\+ \|\la r\ra^{1-s}\la t-r\ra^{-\frac{1}{2}}
 r^{-1}(\partial_t+\partial_r)(rv)\|_{\ell^\infty_U\ell^2_\tau L^2_{t,x}(C^U_\tau)}
 +\|\la r\ra^{\frac{1}{2}-s} \ang v\|_{L^2_{t,x}}
 \\\lesssim \|\partial^{\le 1} v(0,\cd)\|_{L^2} + \inf_{F=F_1+F_2}\Bigl(\|\la
 r\ra^{\frac{3}{2}-s+\frac{a}{2}} F_1\|_{L^2_{t,x}}
 + \|\la r\ra^{1-s+\frac{a}{2}}\la t-r\ra^{\frac{1}{2}} F_2\|_{\ell^1_U
   \ell^2_{\tau} L^2_{t,x}(C^U_\tau)}\Bigr).
\end{multline}
\end{proposition}

\begin{proof}
We examine
\[  \int_0^T\int_{\R^3} \la r\ra^{2-2s} e^{-\sigma_U(t-r)}\Bigl(\Box v +
  \frac{a}{\la r\ra}\beta(r)\Bigl(\partial_t+\partial_r+\frac{1}{r}\Bigr)v -
  \frac{a(a+2)}{4}\beta(r)\frac{1}{\la r\ra^2}v\Bigr)
  \Bigl(\partial_t+\partial_r+\frac{1}{r}\Bigr)v\,dx\,dt.\]
Expanding into spherical coordinates and conjugating, this is
equivalent to
\begin{equation}\label{rp1}
 \int_0^T\int_{\S^2}\int_0^\infty
  \la r\ra^{2-2s}e^{-\sigma_U(t-r)}\Bigl(\partial_t^2-\partial_r^2-\ang\cdot\ang +
  \frac{a}{\la r\ra}\beta(r)(\partial_t+\partial_r)-\frac{a(a+2)}{4}\beta(r)\frac{1}{\la
    r\ra^2}\Bigr)(rv) (\partial_t+\partial_r)(rv)\,dr\,d\omega\,dt. 
\end{equation}
Applying the chain rule, integration by parts, and the fact that
\[[\ang,\partial_r]=\frac{1}{r}\ang,\]
\eqref{rp1} is the same as
\begin{multline}
  \label{rp2}
 \frac{1}{2}\int_0^T\int_{\S^2}\int_0^\infty \la r\ra^{2-2s}e^{-\sigma_U(t-r)}
 (\partial_t-\partial_r)[(\partial_t+\partial_r)(rv)]^2\,dr\,d\omega\,dt
\\ +\int_0^T\int_{\S^2}\int_0^\infty \la r\ra^{2-2s} e^{-\sigma_U(t-r)} \frac{1}{r}
 |\ang(rv)|^2\,dr\,d\omega\,dt
 \\+\frac{1}{2} \int_0^T\int_{\S^2}\int_0^\infty \la r\ra^{2-2s}e^{-\sigma_U(t-r)}
 (\partial_t+\partial_r)|\ang(rv)|^2\,dr\,d\omega\,dt
\\ +a \int_0^T\int_{\S^2}\int_0^\infty \beta(r)\la r\ra^{1-2s}e^{-\sigma_U(t-r)}
 [(\partial_t+\partial_r)(rv)]^2\,dr\,d\omega\,dt
 \\-\frac{a(a+2)}{8} \int_0^T\int_{\S^2}\int_0^\infty \beta(r)\la r\ra^{-2s}e^{-\sigma_U(t-r)}(\partial_t+\partial_r)(rv)^2\,dr\,d\omega\,dt.
\end{multline}
Integration by parts gives that the above is equal to
\begin{multline}
  \label{rp3}
  \frac{1}{2}\int_{\R^3} \la r\ra^{2-2s} e^{-\sigma_U(t-r)}\Bigl\{
  [r^{-1}(\partial_t+\partial_r)(rv)]^2+|\ang v|^2\Bigr\}\,dx\Bigl|_{t=0}^T
  +\frac{1}{2}\int_0^T\int_{\S^2} e^{-\sigma_U(t)}
  (v(t,0))^2\,d\omega\,dt
  \\+(1-s)\int_0^T\int_{\R^3}(1-\beta)(r) \la r\ra^{1-2s}e^{-\sigma_U(t-r)}
  [r^{-1}(\partial_t+\partial_r)(rv)]^2\,dx\,dt
  \\+(1-s+a) \int_0^T\int_{\R^3} \beta(r)\la r\ra^{1-2s}e^{-\sigma_U(t-r)}
  [r^{-1}(\partial_t+\partial_r)(rv)]^2\,dx\,dt
  \\+\int_0^T\int_{\R^3} \la r\ra^{2-2s} \sigma_U'(t-r)
  e^{-\sigma_U(t-r)} [r^{-1}(\partial_t+\partial_r)(rv)]^2\,dx\,dt
  \\+\int_0^T\int_{\R^3} \la r\ra^{2-2s} \Bigl(\frac{1}{r} -
  \frac{1-s}{\la r\ra}\Bigr)  e^{-\sigma_U(t-r)} |\ang v|^2\,dx\,dt
   \\+\frac{a(a+2)}{8}\int_0^T\int_{\R^3} \beta'(r)\la r\ra^{-2s}e^{-\sigma_U(t-r)} v^2\,dx\,dt\\
  -\frac{a(a+2)}{8} \int_{\R^3} \beta(r) \la r\ra^{-2s}  e^{-\sigma_U(t-r)} v^2\,dx\Bigl|_{t=0}^T
  -\frac{sa(a+2)}{4} \int_0^T\int_{\R^3} \beta(r)\la r\ra^{-1-2s}
  e^{-\sigma_U(t-r)} v^2\,dx\,dt.
\end{multline}
Taking a supremum in $T$ and using \eqref{HardyR} to bound the last
two terms in \eqref{rp3}, we get
\begin{multline}
  \label{rp4}
  \frac{1}{2} \sup_{t} \int_{\R^3} \la r\ra^{2-2s} e^{-\sigma_U(t-r)}
  \Bigl\{[r^{-1}(\partial_t+\partial_r)(rv)(t,x)]^2+ |\ang
  v(t,x)|^2\Bigr\}\,dx
  \\+ \frac{(2s-a)(2-2s+a)}{4s} \int_0^\infty\int_{\R^3} \beta(r) \la r\ra^{1-2s} e^{-\sigma_U(t-r)}
  [r^{-1}(\partial_t+\partial_r)(rv)]^2\,dx\,dt
  \\+(1-s)\int_0^\infty \int_{\R^3} (1-\beta)(r) \la
  r\ra^{1-2s}e^{-\sigma_U(t-r)} [r^{-1}(\partial_t+\partial_r)(rv)]^2\,dx\,dt
\\  +\int_0^\infty \int_{\R^3}\la r\ra^{2-2s} e^{-\sigma_U(t-r)}\sigma'_U(t-r) [r^{-1}(\partial_t+\partial_r)(rv)]^2\,dx\,dt
  +s\int_0^\infty\int_{\R^3}\la r\ra^{1-2s}  e^{-\sigma_U(t-r)}|\ang v|^2\,dx\,dt
  \\\le
  \frac{1}{2} \int_{\R^3} \la r\ra^{2-2s} e^{-\sigma_U(-r)}
  \Bigl\{[r^{-1}(\partial_t+\partial_r)(rv)(0,x)]^2+ |\ang
  v(0,x)|^2+ \frac{a(a+2)}{2}\la r\ra^{-2} v^2(0,x)\Bigr\}\,dx
  \\+ \int_0^\infty\int_{\R^3} \la r\ra^{2-2s}
  e^{-\sigma_U(t-r)}e^{\frac{a}{2}\phi}|F| |r^{-1}(\partial_t+\partial_r)(rv)|\,dx\,dt
  \\+\int_0^\infty \int_{\R^3} \la
  r\ra^{2-2s}e^{-\sigma_U(t-r)}\Bigl(\frac{a(1-\beta)}{\la
    r\ra}|\partial_t v| + \frac{a}{2\la r\ra} |\beta'| |v| +
  \frac{a^2\beta(1-\beta)}{4\la r\ra^2}|v|
  \Bigr) |r^{-1}(\partial_t+\partial_r)(rv)|\,dx\,dt
  \\+\frac{a(a+2)}{8}\int_0^\infty\int_{\R^3} \beta'(r)\la r\ra^{-2s}e^{-\sigma_U(t-r)}v^2\,dx\,dt.
\end{multline}
Given the range of parameters, which additionally imply $2s<2+a$, the
coefficients of the terms in the left side are positive.
For the second to last term in \eqref{rp4}, we can use the Schwarz
inequality and Young's inequality to absorb the
$r^{-1}(\partial_t+\partial_r)(rv)$ factor back into the left.
We also use the boundedness of $\sigma_U(t-r)$ and that
\[\sigma_U'(t-r)\gtrsim \la t-r\ra^{-1},\quad\text{ on }
  \{\la t-r\ra\approx U\}\]
to see that
\begin{multline}
  \label{rp5}
  \sup_{t} \int_{\R^3} \la r\ra^{2-2s} 
  \Bigl\{[r^{-1}(\partial_t+\partial_r)(rv)(t,x)]^2+ |\ang v(t,x)|^2\Bigr\}\,dx
  + \int_0^\infty\int_{\R^3} \la r\ra^{1-2s}
  [r^{-1}(\partial_t+\partial_r)(rv)]^2\,dx\,dt
\\  +\int \int_{\{t-r\approx U\}}\la r\ra^{2-2s} \la t-r\ra^{-1} [r^{-1}(\partial_t+\partial_r)(rv)]^2\,dx\,dt
  +\int_0^\infty\int_{\R^3}\la r\ra^{1-2s} |\ang v|^2\,dx\,dt
  \\\lesssim
\int_{\R^3} \la r\ra^{2-2s} 
  \Bigl\{[r^{-1}(\partial_t+\partial_r)(rv)(0,x)]^2+ |\ang
  v(0,x)|^2 + \frac{a(a+2)}{2}\la r\ra^{-2} v^2(0,x)\Bigr\}\,dx
  \\+ \int_0^\infty\int_{\R^3} \la r\ra^{2-2s}e^{\frac{a}{2}\phi} |F|
  |r^{-1}(\partial_t+\partial_r)(rv)|\,dx\,dt
  + \int_0^\infty\int_{\{|x|\le 2\}} \Bigl(v^2 + (\partial_t v)^2\Bigr)\,dx\,dt.
\end{multline}
The singular behavior at $r=0$ for the Cauchy data can easily be resolved
using \eqref{L2Hardy}.
For the second to last term, we apply the Schwarz and H\"older inequalities and use Young's inequality to absorb
the multiplier factor into the left.  Here we also use \eqref{vw}.  For the last term, we
bound using \eqref{energyZv}.  Using these estimates on the last two terms of \eqref{rp5} yields \eqref{rpd}.
\end{proof}

We now introduce
the higher order variant of \eqref{rpd} and also use \eqref{HardyR}
and \eqref{HardyU} to
introduce bounds on the undifferentiated solution.
\begin{corollary}\label{cor_rp}
  Suppose $a, s$ satisfy
$0<a<2s<1$, and let $0\le \sigma<s$.  Then for $v\in C^4(\R_+\times\R^3)$ solving
\eqref{boxv} subject to \eqref{supp}, we have
\begin{multline}
  \label{rp}
\|\la r\ra^{-\frac{1}{2}-s} Z^{\le 2} v\|_{L^2_{t,x}}
  + \|\la r\ra^{\frac{1}{2}-s}
  r^{-1}(\partial_t+\partial_r)(rZ^{\le 2}v)\|_{L^2_{t,x}}
+\|\la r\ra^{\frac{1}{2}-s} \ang Z^{\le 2} v\|_{L^2_{t,x}}
\\  +\|\la r\ra^{-\frac{1}{2}-s+\sigma}\la t-r\ra^{-\sigma}
Z^{\le 2} v\|_{\ell^{\frac{2}{1-2\sigma}}_U\ell^2_{\tau}L^2_{t,x}(C^U_\tau)}
\\+ \|\la r\ra^{\frac{1}{2}-s+\sigma}\la t-r\ra^{-\sigma}
r^{-1}(\partial_t+\partial_r)(rZ^{\le 2} v)\|_{\ell^{\frac{2}{1-2\sigma}}_U
  \ell^2_{\tau} L^2_{t,x}(C^U_\tau)}
 \\\lesssim \|\partial \nabla^{\le 2} v(0,\cd)\|_{L^2} + \|\la
 r\ra^{\frac{3}{2}-s+\frac{a}{2}} Z^{\le 2} F\|_{\ell^2_\tau\ell^2_{R\le \tau/2} L^2_{t,x}(C^R_\tau)}
 + \|\la r\ra^{1-s+\frac{a}{2}}\la t-r\ra^{\frac{1}{2}} Z^{\le 2} F\|_{\ell^1_U
   \ell^2_{\tau\ge 4U} L^2_{t,x}(C^U_\tau)}.
\end{multline}
\end{corollary}

\begin{proof}
  We combine \eqref{rpd} with \eqref{HardyR} (with $\beta\equiv 1$), \eqref{HardyU}, and
\eqref{HolderU} to get the following.
\begin{multline}
  \label{rpz0}
\|\la r\ra^{-\frac{1}{2}-s} v\|_{L^2_{t,x}}
  + \|\la r\ra^{\frac{1}{2}-s}
  r^{-1}(\partial_t+\partial_r)(rv)\|_{L^2_{t,x}}
+\|\la r\ra^{\frac{1}{2}-s} \ang v\|_{L^2_{t,x}}
\\  +\|\la r\ra^{-\frac{1}{2}-s+\sigma}\la t-r\ra^{-\sigma}
v\|_{\ell^{\frac{2}{1-2\sigma}}_U\ell^2_{\tau}L^2_{t,x}(C^U_\tau)}
+ \|\la r\ra^{\frac{1}{2}-s+\sigma}\la t-r\ra^{-\sigma}
r^{-1}(\partial_t+\partial_r)(rv)\|_{\ell^{\frac{2}{1-2\sigma}}_U
  \ell^2_{\tau} L^2_{t,x}(C^U_\tau)}
 \\\lesssim \|\partial^{\le 1} v(0,\cd)\|_{L^2} + \inf_{F=F_1+F_2}\Bigl(\|\la
 r\ra^{\frac{3}{2}-s+\frac{a}{2}} F_1\|_{L^2_{t,x}}
 + \|\la r\ra^{1-s+\frac{a}{2}}\la t-r\ra^{\frac{1}{2}} F_2\|_{\ell^1_U
   \ell^2_{\tau} L^2_{t,x}(C^U_\tau)}\Bigr).
\end{multline}
As such, applying \eqref{rpz0} to $Z^{\le 1} v$ and using
\eqref{Scomm}, \eqref{Omegacomm}, and \eqref{dcomm}, which are
included in $F_1$, we see that
\begin{multline}
  \label{rpz1}
\|\la r\ra^{-\frac{1}{2}-s} Z^{\le 1}v\|_{L^2_{t,x}}
  + \|\la r\ra^{\frac{1}{2}-s}
  r^{-1}(\partial_t+\partial_r)(rZ^{\le 1}v)\|_{L^2_{t,x}}
  + \|\la r\ra^{\frac{1}{2}-s} \ang Z^{\le 1} v\|_{L^2_{t,x}}
\\  +\|\la r\ra^{-\frac{1}{2}-s+\sigma}\la t-r\ra^{-\sigma}
Z^{\le 1} v\|_{\ell^{\frac{2}{1-2\sigma}}_U\ell^2_{\tau}L^2_{t,x}(C^U_\tau)}
+ \|\la r\ra^{\frac{1}{2}-s+\sigma}\la t-r\ra^{-\sigma}
r^{-1}(\partial_t+\partial_r)(rZ^{\le 1}v)\|_{\ell^{\frac{2}{1-2\sigma}}_U
  \ell^2_\tau L^2_{t,x}(\C^U_\tau)}
 \\\lesssim \|\partial \nabla^{\le 1} v(0,\cd)\|_{L^2} + \inf_{F=F_1+F_2}\Bigl(\|\la
 r\ra^{\frac{3}{2}-s+\frac{a}{2}} Z^{\le 1} F_1\|_{ L^2_{t,x}}
 + \|\la r\ra^{1-s+\frac{a}{2}}\la t-r\ra^{\frac{1}{2}} Z^{\le 1} F_2\|_{\ell^1_U
   \ell^2_\tau L^2_{t,x}(C^U_\tau)}\Bigr)
 \\+ \|\la r\ra^{-\frac{1}{2}-s}
 r^{-1}(\partial_t+\partial_r)(r v)\|_{L^2_{t,x}}
 +\|\la r\ra^{-\frac{1}{2}-s} \ang v\|_{L^2_{t,x}}
   + \|\la r\ra^{-\frac{3}{2}-s} v\|_{L^2_{t,x}}
   + \|\partial^{\le 1} v\|_{L^2_{t,x}(\R_+\times\{r\le 2\})}.
 \end{multline}
 Subsequently applying \eqref{rpz1} to $Z^{\le 1} v$ and taking care
 with the limited regularity near $r=0$ of $\omega_j$, we get
\begin{multline}
  \label{rpz2}
\|\la r\ra^{-\frac{1}{2}-s} Z^{\le 2}v\|_{L^2_{t,x}}
  + \|\la r\ra^{\frac{1}{2}-s}
  r^{-1}(\partial_t+\partial_r)(rZ^{\le 2}v)\|_{L^2_{t,x}}
  + \|\la r\ra^{\frac{1}{2}-s} \ang Z^{\le 2} v\|_{L^2_{t,x}}
\\  +\|\la r\ra^{-\frac{1}{2}-s+\sigma}\la t-r\ra^{-\sigma}
Z^{\le 2} v\|_{\ell^{\frac{2}{1-2\sigma}}_U\ell^2_{\tau}L^2_{t,x}(C^U_\tau)}
+ \|\la r\ra^{\frac{1}{2}-s+\sigma}\la t-r\ra^{-\sigma}
r^{-1}(\partial_t+\partial_r)(rZ^{\le 2}v)\|_{\ell^{\frac{2}{1-2\sigma}}_U
  \ell^2_\tau L^2_{t,x}(\C^U_\tau)}
 \\\lesssim \|\partial \nabla^{\le 2} v(0,\cd)\|_{L^2} + \inf_{F=F_1+F_2}\Bigl(\|\la
 r\ra^{\frac{3}{2}-s+\frac{a}{2}} Z^{\le 2} F_1\|_{ L^2_{t,x}}
 + \|\la r\ra^{1-s+\frac{a}{2}}\la t-r\ra^{\frac{1}{2}} Z^{\le 2} F_2\|_{\ell^1_U
   \ell^2_\tau L^2_{t,x}(C^U_\tau)}\Bigr)
 \\+ \|\la r\ra^{-\frac{1}{2}-s}
 r^{-1}(\partial_t+\partial_r)(r Z^{\le 1} v)\|_{L^2_{t,x}}
 + \|\la r\ra^{-\frac{1}{2}-s} \ang Z^{\le 1}v\|_{L^2_{t,x}}
   + \|\la r\ra^{-\frac{3}{2}-s} Z^{\le 1} v\|_{L^2_{t,x}}
 \\  + \|\partial^{\le 1} Z^{\le 1} v\|_{L^2_{t,x}(\R_+\times\{r\le
     2\})}
   + \|r^{-1}(1-\beta(r/2))\partial_t v\|_{L^2_{t,x}}.
 \end{multline}
The first three error terms are readily bounded using \eqref{rpz1} and
\eqref{rpz0} to give
\begin{multline}
  \label{rpz2a}
\|\la r\ra^{-\frac{1}{2}-s} Z^{\le 2}v\|_{L^2_{t,x}}
  + \|\la r\ra^{\frac{1}{2}-s}
  r^{-1}(\partial_t+\partial_r)(rZ^{\le 2}v)\|_{L^2_{t,x}}
  + \|\la r\ra^{\frac{1}{2}-s} \ang Z^{\le 2} v\|_{L^2_{t,x}}
\\  +\|\la r\ra^{-\frac{1}{2}-s+\sigma}\la t-r\ra^{-\sigma}
Z^{\le 2} v\|_{\ell^{\frac{2}{1-2\sigma}}_U\ell^2_{\tau}L^2_{t,x}(C^U_\tau)}
+ \|\la r\ra^{\frac{1}{2}-s+\sigma}\la t-r\ra^{-\sigma}
r^{-1}(\partial_t+\partial_r)(rZ^{\le 2}v)\|_{\ell^{\frac{2}{1-2\sigma}}_U
  \ell^2_\tau L^2_{t,x}(\C^U_\tau)}
 \\\lesssim \|\partial \nabla^{\le 2} v(0,\cd)\|_{L^2} + \inf_{F=F_1+F_2}\Bigl(\|\la
 r\ra^{\frac{3}{2}-s+\frac{a}{2}} Z^{\le 2} F_1\|_{ L^2_{t,x}}
 + \|\la r\ra^{1-s+\frac{a}{2}}\la t-r\ra^{\frac{1}{2}} Z^{\le 2} F_2\|_{\ell^1_U
   \ell^2_\tau L^2_{t,x}(C^U_\tau)}\Bigr)
 \\   + \|\partial^{\le 1} Z^{\le 1} v\|_{L^2_{t,x}(\R_+\times\{r\le
     2\})}
   + \|r^{-1}(1-\beta(r/2))\partial_t v\|_{L^2_{t,x}}.
 \end{multline}
 As in the proof of Corollary \ref{cor_energyZv}, we apply
 \eqref{L2Hardy} to the last term, which gives that the remaining
 error terms are
 \[
\|\partial^{\le 1} Z^{\le 1} v\|_{L^2_{t,x}(\R_+\times\{r\le
     2\})}
   + \|r^{-1}(1-\beta(r/2)) \partial_t v\|_{L^2_{t,x}}
\lesssim \|\partial^{\le 1} Z^{\le 1} v\|_{L^2_{t,x}(\R_+\times \{r\le 4\})}.\]
The right side can be controlled using \eqref{energyZv}.

Finally, we make the choice of decomposition $F=F_1+F_2$
 \[F_1=\bigcup_\tau \bigcup_{1\le R\le \tau/2} \chi_{C^R_\tau}
   Z^{\le 2} F,\qquad F_2 = \bigcup_{\tau} \bigcup_{1\le U\le \tau/4}
  \chi_{C^U_\tau}Z^{\le 2} F.\]
\end{proof}

\vskip .25in
\section{Space-time Klainerman Sobolev estimates}
We use the variant of the space-time Klainerman-Sobolev of \cite{MTT}
that was introduced in \cite{AMW}.  The key was to use a different
change of variables so that the differentiated term was necessarily a
good derivative $\partial_t+\partial_r+\frac{1}{r}$ and did not
necessitate introducing a third vector field to change $\partial_r$
into such a good derivative.  The theorem of \cite{AMW} states:
\begin{theorem}\label{thm_ks}
  For $\tau\ge 1$, $1\le R\le \tau/2$, $1\le U\le \tau/4$, we have
  \begin{multline}
    \label{CRT}
    \|w\|_{L^\infty_{t,x}(C^R_\tau)} \lesssim
    \frac{1}{\tau^{\frac{1}{2}}R^{\frac{3}{2}}} \|Z^{\le
        2}w\|_{L^2_{t,x}(\tC^R_\tau)} +
      \frac{1}{\tau^{\frac{1}{2}}R^{\frac{1}{2}}}
        \|r^{-1}(\partial_t+\partial_r)(rZ^{\le 2}
        w)\|_{L^2_{t,x}(\tC^R_\tau)}
        \\+ \frac{1}{\tau^{\frac{1}{2}}R^{\frac{1}{2}}} \|\ang (Z^{\le
          2} w)\|_{L^2_{t,x}(\tC^R_\tau)}
      \end{multline}
      and
      \begin{equation}
        \label{CUT}
        \|w\|_{L^\infty_{t,x}(C^U_\tau)} \lesssim
        \frac{1}{\tau^{\frac{3}{2}}U^{\frac{1}{2}}} \|Z^{\le
          2}w\|_{L^2_{t,x}(\tC^U_\tau)} +
        \frac{1}{\tau^{\frac{1}{2}}U^{\frac{1}{2}}}
        \|r^{-1}(\partial_t+\partial_r)(rZ^{\le 2} w)\|_{L^2_{t,x}(\tC^U_\tau)}.
      \end{equation}
      Moreover, if $2\le q\le 4$, we have
      \begin{multline}
        \label{CRTq}
\|w\|_{L^q_{t,x}(C^R_\tau)} \lesssim \tau^{\frac{1}{q}-\frac{1}{2}}
R^{\frac{3}{q}-\frac{3}{2}} \|Z^{\le 1} w\|_{L^2_{t,x}(\tC^R_\tau)}
+ \tau^{\frac{1}{q}-\frac{1}{2}} R^{\frac{3}{q}-\frac{1}{2}}
\|r^{-1}(\partial_t+\partial_r)(rw)\|_{L^2_{t,x}(\tC^R_\tau)}
\\+\tau^{\frac{1}{q}-\frac{1}{2}}R^{\frac{3}{q}-\frac{1}{2}} \|\ang w\|_{L^2_{t,x}(\tC^R_\tau)}
\end{multline}
and
\begin{equation}
  \label{CUTq}
  \|w\|_{L^q_{t,x}(C^U_\tau)} \lesssim U^{\frac{1}{q}-\frac{1}{2}}
  \tau^{\frac{3}{q}-\frac{3}{2}} \|Z^{\le 1}
  w\|_{L^2_{t,x}(\tC^U_\tau)}
  + U^{\frac{1}{q}-\frac{1}{2}}\tau^{\frac{3}{q}-\frac{1}{2}}\|r^{-1}(\partial_t+\partial_r)(rw)\|_{L^2_{t,x}(\tC^U_\tau)}.
\end{equation}
\end{theorem}

For the majority of the cases ($R>1$, $U>1$), we apply Sobolev embeddings in
$(s,\rho,\omega)$ where
\[t-r=e^s,\qquad r=e^{s+\rho}.\]
If we set $\phi(s,\rho,\omega)=(e^s+e^{s+\rho},e^{s+\rho}\omega)$,
the key observation is that
\[\partial_s (w(\phi(s,\rho,\omega))) = (Sw)(\phi(s,\rho,\omega)),
  \qquad
  \partial_\rho(w(\phi(s,\rho,\omega))) = (r(\partial_t+\partial_r)w)(\phi(s,\rho,\omega)).\]

For $U=1$, we instead use $(q,\rho,\omega)$ with
\[t=q+e^\rho,\qquad r=e^\rho.\]
And the most delicate case is $R=1$, where we apply standard
embeddings in $(s,y)$ with
\[t=e^s+|y|,\qquad x=y.\]
In this last case extra care must be taken with singularities near the
origin, and a version of \eqref{HardyR} is used.

\vskip .25in
\section{Proof of Theorem \ref{main_thm}}
To show that \eqref{conjugated} has a global solution provided
$\varepsilon>0$ is sufficiently small, we set $v_0\equiv 0$ and let
$v_j$ solve
\begin{equation}
  \label{vj}
  \begin{cases}
    \L v_j = e^{\frac{a}{2}\phi(r)}F_p(e^{-\frac{a}{2}\phi(r)}v_{j-1}),\\
    v_j(0,\cd) = \varepsilon v_{(0)},\quad \partial_t
    v_j(0,\cd)=\varepsilon v_{(1)}.
  \end{cases}
\end{equation}
Here $v_{(0)}, v_{(1)}\in C^\infty_c(\R^3)$, and without loss of
generality, we shall assume that they have supports contained in
$\{r\le \frac{1}{2}\}$.

We fix $\max\Bigl(p_c(3+a), \frac{7}{3}\Bigr)<p<\frac{4}{2+a}+1$.  We
notice that $p_c(3+a)\ge\frac{7}{3}$ provided $a\le \frac{1}{7}$. 
And we set
\[s=\frac{3+a}{2}-\frac{2}{p-1},\quad \sigma=\frac{p-2}{2p}.\]
With this choice of $\sigma$, we have $\frac{2}{1-2\sigma}=p$.
We make a few notes about the hypotheses in Corollary \ref{cor_rp}.
The requirement that $2s<1$ corresponds to $p<\frac{4}{2+a}+1$.  The
latter value, when $a=0$, corresponds to the conformal exponent
for the wave equation, and global existence for powers greater than or
equal to this value is known to follow from Strichartz estimates.
Additionally
\[a<2s \iff p>\frac{7}{3},\]
which is the source of the latter restriction.  We also observe 
\[\sigma<s \iff p>p_c(3+a).\]

We will bound and show convergence in the space dictated by:
\begin{multline}
  \label{X}
  \|v\|_X = 
  \|\la r\ra^{-\frac{1}{2}-s} v\|_{L^2_{t,x}}
 + \|\la r\ra^{\frac{1}{2}-s}
 r^{-1}(\partial_t+\partial_r)(rv)\|_{L^2_{t,x}}
 + \|\la r\ra^{\frac{1}{2}-s}\ang v\|_{L^2_{t,x}}
\\+\|\la r\ra^{-\frac{1}{2}-s+\sigma}\la t-r\ra^{-\sigma}
v\|_{\ell^{p}_U\ell^2_{\tau}L^2_{t,x}(C^U_\tau)}
+ \|\la r\ra^{\frac{1}{2}-s+\sigma}\la t-r\ra^{-\sigma}
r^{-1}(\partial_t+\partial_r)(rv)\|_{\ell^{p}_U
  \ell^2_\tau L^2_{t,x}(\C^U_\tau)},
\end{multline}
along with its higher order variants $\|Z^{\le 2} v\|_X$.

To complete the proof, we shall first examine the multilinear bounds
needed to close the iteration.  These are similar to what appeared in
\cite{AMW}, but we must account for the additional weights.
Subsequently, we will show that there exists a $C_0$ so that
\begin{equation}
  \label{bdd_goal}
  \|Z^{\le 2}v_j\|_X \le 2C_0\varepsilon.
\end{equation}
In the last section, we show that the sequence $(v_j)$ is Cauchy in
$X$.  It must, thus, converge, and by standard results, the limit
corresponds to the desired solution.

\vskip .125in
\subsection{Multilinear estimates}
Based on the right sides of \eqref{rp}
and on the possibility of $Z^{\le 2} F_p(v)$ contributing $|v|^{p-1}Z^{\le 2}
v$ and $|v|^{p-2} (Z^{\le 1}v)^2$, we consider four cases:
\begin{align}
  \label{Z2U}
\|\la r\ra^{1-s-\frac{(p-1)a}{2}} \la
  t-r\ra^{\frac{1}{2}}|g|^{p-1}f\|_{\ell^1_U\ell^2_{\tau\ge 4U}
  L^2_{t,x}(C^U_\tau)}&\lesssim \|Z^{\le 2} g\|^{p-1}_X \|f\|_X,\\
  \label{Z2R}
  \|\la r\ra^{\frac{3}{2}-s-\frac{(p-1)a}{2}} |g|^{p-1} f\|_{\ell^2_\tau
  \ell^2_{R\le \tau/2} L^2_{t,x}(C^R_\tau)}&\lesssim \|Z^{\le 2}
  g\|^{p-1}_X \|f\|_X,\\
  \label{Z1U}
\|\la r\ra^{1-s-\frac{(p-1)a}{2}} \la t-r\ra^{\frac{1}{2}}
  |g|^{p-2}f^2\|_{\ell^1_U\ell^2_{\tau\ge 4U}L^2_{t,x}(C^U_\tau)}&\lesssim \|Z^{\le 2}g\|^{p-2}_X
                                \|Z^{\le 1} f\|^2_X,
  \\
  \label{Z1R} \|\la r\ra^{\frac{3}{2}-s-\frac{(p-1)a}{2}} |g|^{p-2}
  f^{2}\|_{\ell^2_\tau\ell^2_{R\le \tau/2} L^2_{t,x}(C^R_\tau)}&\lesssim \|Z^{\le 2} g\|_X^{p-2} \|Z^{\le 1} f\|_X^2.
\end{align}
These are all quite analogous to the corresponding estimates in
\cite{AMW}, though we must take care with the additional weights $\la
r\ra^{-a/2}$.  In each case, we apply H\"older's inequality followed
by the bounds from Theorem \ref{thm_ks}.

\vskip .125in
\subsubsection{Proof of \eqref{Z2U}}
We begin with
\begin{multline*}\|\la r\ra^{1-s-\frac{(p-1)a}{2}} \la t-r\ra^{\frac{1}{2}} |g|^{p-1}
  f\|_{L^2_{t,x}(C^U_\tau)}
  \\\lesssim
  \|\la r\ra^{\frac{3-2\sigma}{2(p-1)}-\frac{a}{2}} \la t-r\ra^{\frac{1+2\sigma}{2(p-1)}}  g\|^{p-1}_{L^\infty_{t,x}(C^U_\tau)}
\|\la r\ra^{-\frac{1}{2}-s+\sigma} \la t-r\ra^{-\sigma} f\|_{L^2_{t,x}(C^U_\tau)}.
\end{multline*}
For $U\le \tau/4$, we may apply \eqref{CUT}.
Noticing that
\[\frac{3-2\sigma}{2(p-1)} - \frac{3}{2} = -\frac{1}{2}
  -\Bigl(\frac{3}{2}-\frac{2}{p-1}\Bigr) +\sigma
\qquad\text{and}\qquad \frac{1+2\sigma}{2(p-1)} = \frac{1}{2}-\sigma
  ,\]
we obtain that the right side
is in turn bounded by
\begin{multline*}
\Bigl(\|\la r\ra^{-\frac{1}{2}-s+\sigma} \la t-r \ra^{-\sigma} Z^{\le 2} g\|_{L^2_{t,x}(\tC^U_\tau)}
+ \|\la r\ra^{\frac{1}{2}-s+\sigma} \la
t-r\ra^{-\sigma}r^{-1}(\partial_t+\partial_r)(rZ^{\le 2} g)\|_{L^2_{t,x}(\tC^U_\tau)}
\Bigr)^{p-1}
\\\times
\|\la r\ra^{-\frac{1}{2}-s+\sigma} \la t-r\ra^{-\sigma} f\|_{L^2_{t,x}(C^U_\tau)}.  
\end{multline*}
Thus, upon summing and using that $2\le 2p$, we get
\begin{multline*}
  \|\la r\ra^{1-s-\frac{(p-1)a}{2}} \la t-r\ra^{\frac{1}{2}} |g|^{p-1}
  f\|_{\ell^1_U \ell^2_{\tau\ge 4U} L^2_{t,x}(C^U_\tau)}
  \\\lesssim \Bigl(\|\la r\ra^{-\frac{1}{2}-s+\sigma} \la t-r
  \ra^{-\sigma} Z^{\le 2} g\|^{p-1}_{\ell^p_U \ell^2_{\tau\ge 4U} L^2_{t,x}(\tC^U_\tau)}
  +
  \|\la r\ra^{\frac{1}{2}-s+\sigma} \la
t-r\ra^{-\sigma}r^{-1}(\partial_t+\partial_r)(rZ^{\le 2}
g)\|^{p-1}_{\ell^p_U \ell^2_{\tau\ge 4U}L^2_{t,x}(\tC^U_\tau)}
  \Bigr)
 \\ \times \|\la r\ra^{-\frac{1}{2}-s+\sigma} \la t-r\ra^{-\sigma}
  f\|_{\ell^p_U \ell^2_{\tau\ge 4U}L^2_{t,x}(C^U_\tau)},
\end{multline*}
which is $\O(\|Z^{\le 2}g\|_X^{p-1}\|f\|_X)$ as claimed.\qed

\vskip .125in
\subsubsection{Proof of \eqref{Z2R}}
Observing that
\[  \|\la r\ra^{\frac{3}{2}-s-\frac{(p-1)a}{2}} |g|^{p-1}
  f\|_{L^2_tL^2_x(C^R_\tau)}
  \lesssim \|\la r \ra^{\frac{2}{p-1}-\frac{a}{2}}
  g\|^{p-1}_{L^\infty_{t,x}(C^R_\tau)}
  \|\la r\ra^{-\frac{1}{2}-s} f\|_{L^2_{t,x}(C^R_\tau)}
\]
and applying \eqref{CRT} to the first factor in the right, we get
\begin{multline*}
  \|\la r\ra^{\frac{3}{2}-s-\frac{(p-1)a}{2}} |g|^{p-1}
  f\|_{L^2_tL^2_x(C^R_\tau)}\lesssim
  \|\la r\ra^{-\frac{1}{2}-s} f\|_{L^2_{t,x}(C^R_\tau)}
  \\\times\Bigl(\|\la r\ra^{-\frac{1}{2}-s}Z^{\le 2}
  g\|_{L^2_{t,x}(\tC^R_\tau)}
+ \|\la r\ra^{\frac{1}{2}-s}r^{-1}(\partial_t+\partial_r)(rZ^{\le
  2}g)\|_{L^2_{t,x}(\tC^R_\tau)}
+\|\la r\ra^{\frac{1}{2}-s} \ang Z^{\le 2} g\|_{L^2_{t,x}(\tC^R_\tau)}
  \Bigr)^{p-1}
\end{multline*}
provided that $1\le R\le \tau/2$.  Thus, upon summing in $R, \tau$,
\eqref{Z2R} follows.\qed

\vskip .125in
\subsubsection{Proof of \eqref{Z1U}}
We now observe that H\"older's inequality gives
\[\|\la r\ra^{1-s-\frac{(p-1)a}{2}}\la t-r\ra^{\frac{1}{2}} |g|^{p-2}
  f^2\|_{L^2_{t,x}(C^U_\tau)}
  \lesssim \tau^{1-s-\frac{(p-1)a}{2}} U^{\frac{1}{2}}
  \|g\|^{p-2}_{L^\infty_t L^\infty_x(C^U_\tau)}
  \|f\|^2_{L^4_tL^4_x(C^U_\tau)}.\]
We assume $1\le U\le \tau/4$.
To the factors in the right side, we apply \eqref{CUT} and
\eqref{CUTq} respectively, which bounds the above by
\begin{multline*}
  \Bigl(\|\la r\ra^{-\frac{1}{2}-s+\sigma} \la t-r\ra^{-\sigma} Z^{\le
    2} g\|_{L^2_{t,x}(\tC^U_\tau)} +
\|\la r\ra^{\frac{1}{2}-s+\sigma} \la t-r\ra^{-\sigma} r^{-1}(\partial_t+\partial_r)(rZ^{\le
    2} g)\|_{L^2_{t,x}(\tC^U_\tau) }
    \Bigr)^{p-2}
  \\\times \Bigl(\|\la r\ra^{-\frac{1}{2}-s+\sigma} \la
  t-r\ra^{-\sigma} Z^{\le 1} f\|_{L^2_{t,x}(\tC^U_\tau)} +\|\la r\ra^{\frac{1}{2}-s+\sigma} \la
  t-r\ra^{-\sigma} r^{-1}(\partial_t+\partial_r)(rf)\|_{L^2_{t,x}(\tC^U_\tau)}\Bigr)^2.
\end{multline*}
It then follows that
\begin{multline*}
  \|\la r\ra^{1-s-\frac{(p-1)a}{2}}\la t-r\ra^{\frac{1}{2}} |g|^{p-2}
  f^2\|_{\ell^1_U \ell^2_{\tau\ge 4U}L^2_{t,x}(C^U_\tau)}
  \\\lesssim
  \Bigl(\|\la r\ra^{-\frac{1}{2}-s+\sigma} \la t-r\ra^{-\sigma} Z^{\le
    2} g\|^{p-2}_{\ell^p_U \ell^2_{\tau\ge 4U}L^2_{t,x}(\tC^U_\tau)} +
\|\la r\ra^{\frac{1}{2}-s+\sigma} \la t-r\ra^{-\sigma} r^{-1}(\partial_t+\partial_r)(rZ^{\le
    2} g)\|^{p-2}_{\ell^p_U \ell^2_{\tau\ge 4U}L^2_{t,x}(\tC^U_\tau) }
    \Bigr)
  \\\times \Bigl(\|\la r\ra^{-\frac{1}{2}-s+\sigma} \la
  t-r\ra^{-\sigma} Z^{\le 1} f\|^2_{\ell^p_U \ell^2_{\tau\ge 4U}L^2_{t,x}(\tC^U_\tau)} +\|\la r\ra^{\frac{1}{2}-s+\sigma} \la
  t-r\ra^{-\sigma} r^{-1}(\partial_t+\partial_r)(rf)\|^2_{\ell^p_U \ell^2_{\tau\ge 4U}L^2_{t,x}(\tC^U_\tau)}\Bigr),
\end{multline*}
which provides the $\O(\|Z^{\le 2} g\|^{p-2}_X \|Z^{\le 1} f\|^2_X)$
bound that we sought.\qed

\vskip .125in
\subsubsection{Proof of \eqref{Z1R}}
We argue exactly as in the proof of \eqref{Z1U} and replace the
application of \eqref{CUT} and \eqref{CUTq} with \eqref{CRT} and \eqref{CRTq} to see
\begin{multline*}
  \|\la r\ra^{\frac{3}{2}-s-\frac{(p-1)a}{2}}
|g|^{p-2}f^2\|_{\ell^2_\tau \ell^2_{R\le \tau/2}L^2_{t,x}(C^R_\tau)}
\\\lesssim\Bigl(\|\la r\ra^{-\frac{1}{2}-s} Z^{\le
  2}g\|_{L^2_{t,x}}^{p-2}
+ \|\la r\ra^{\frac{1}{2}-s} r^{-1}(\partial_t+\partial_r)(rZ^{\le
  2}g)\|_{L^2_{t,x} }^{p-2}
+\|\la r\ra^{\frac{1}{2}-s} \ang Z^{\le 2} g\|^{p-2}_{L^2_{t,x}}
    \Bigr)
\\\times \Bigl(\|\la r\ra^{-\frac{1}{2}-s}Z^{\le 1}
f\|^2_{L^2_{t,x}} +  \|\la r\ra^{\frac{1}{2}-s}r^{-1}
(\partial_t+\partial_r)(rZ^{\le 1} f)\|^2_{L^2_{t,x}}
+\|\la r\ra^{\frac{1}{2}-s} \ang Z^{\le 1} f\|^2_{L^2_{t,x}}
\Bigr),
\end{multline*}
which is $\O(\|Z^{\le 2} g\|_X^{p-2}\|Z^{\le 1} f\|_X^2)$ as
required.  Here we have also used that $\|\cd\|_{\ell^{2p}} \lesssim
\|\cd \|_{\ell^2}$ repeatedly.
\qed

\vskip .125in
\subsection{Boundedness}
We now seek to show \eqref{bdd_goal}, which we shall do via
induction. By \eqref{rp}, there exists
$C_0$ so that
\[\|Z^{\le 2} v_1\|_X \le C_0\varepsilon\]
and
\begin{multline}\label{bdd1}\|Z^{\le 2} v_{j+1}\|_X \le C_0 \varepsilon + C\|\la
  r\ra^{\frac{3}{2}-s+\frac{a}{2}} Z^{\le 2} (F_p(e^{-\frac{a}{2}\phi(r)} v_j))\|_{\ell^2_\tau\ell^2_{R\le \tau/2}
  L^2_{t,x}(C^R_\tau)}
\\+ C\|\la r\ra^{1-s+\frac{a}{2}} \la t-r\ra^{\frac{1}{2}}Z^{\le 2}(F_p(e^{-\frac{a}{2}\phi(r)} v_j))\|_{\ell^1_U
    \ell^2_{\tau\ge 4U}L^2_{t,x}(C^U_\tau)}.
\end{multline}

The proof of \eqref{bdd_goal} will be complete if we can use
\eqref{bdd_goal} (as the inductive hypothesis) to show that
\begin{multline}
  \label{bdd_goal1}
  \|\la
  r\ra^{\frac{3}{2}-s+\frac{a}{2}} Z^{\le 2}(F_p(e^{-\frac{a}{2}\phi(r)} v_j))\|_{\ell^2_\tau\ell^2_{R\le \tau/2}
  L^2_{t,x}(C^R_\tau)}
\\+ \|\la r\ra^{1-s+\frac{a}{2}} \la t-r\ra^{\frac{1}{2}}Z^{\le 2}(F_p(e^{-\frac{a}{2}\phi(r)} v_j))\|_{\ell^1_U
    \ell^2_{\tau\ge 4U}L^2_{t,x}(C^U_\tau)} \lesssim \varepsilon^p,
\end{multline}
which then, when combined with \eqref{bdd1}, gives \eqref{bdd_goal}
provided that $\varepsilon>0$ is sufficiently small.

From \eqref{Fp} and \eqref{vw}, we get
\[|Z^{\le 2} \Bigl(F_p(e^{-\frac{a}{2}\phi(r)} v_j)\Bigr)|\lesssim
\la r\ra^{-\frac{ap}{2}}\Bigl( |v_j|^{p-1} |Z^{\le 2} v_j| +
|v_j|^{p-2}|Z^{\le 1} v_j|^2\Bigr).
\]
Thus, the left side of \eqref{bdd_goal1} is bounded by
\begin{multline}
  \label{bdd2}
\|\la r\ra^{1-s-\frac{(p-1)a}{2}} \la t-r\ra^{\frac{1}{2}}|v_j|^{p-1} Z^{\le 2}v_j\|_{\ell^1_U
    \ell^1_{\tau\ge 4U}L^2_{t,x}(C^U_\tau)}
+  
    \|\la
  r\ra^{\frac{3}{2}-s-\frac{(p-1)a}{2}} |v_j|^{p-1} Z^{\le 2} v_j \|_{\ell^2_\tau\ell^2_{R\le \tau/2}
  L^2_{t,x}(C^R_\tau)}
\\+ \|\la r\ra^{1-s-\frac{(p-2)a}{2}} \la t-r\ra^{\frac{1}{2}}|v_j|^{p-2} (Z^{\le 1}v_j)^2\|_{\ell^1_U
    \ell^1_{\tau\ge 4U}L^2_{t,x}(C^U_\tau)}
+  
    \|\la
  r\ra^{\frac{3}{2}-s-\frac{(p-2)a}{2}} |v_j|^{p-2} (Z^{\le 1} v_j)^2 \|_{\ell^2_\tau\ell^2_{R\le \tau/2}
  L^2_{t,x}(C^R_\tau)}.
\end{multline}
We apply \eqref{Z2U}, \eqref{Z2R}, \eqref{Z1U}, and \eqref{Z1R}
to each of the four terms respectively.  For the first two terms, we
select $g=v_j$, $f=Z^{\le 2}v_j$, while for the latter two terms we
use $g=v_j$, $f=Z^{\le 1}v_j$.  It immediately follows that the left
side of \eqref{bdd_goal1} is $\O(\|Z^{\le 2}v_j\|^p_X)$, which by the
hypothesis yields \eqref{bdd_goal1}.

\vskip .125in
\subsection{Convergence}
It will follow that $(v_j)$ is Cauchy if we can prove that
\begin{equation}
  \label{Cauchy}
  \|v_{j+1}-v_j\|_X \le \frac{1}{2}\|v_j-v_{j-1}\|_X,\quad \forall j\ge
  1.
\end{equation}
Since \eqref{Fp} and \eqref{vw} gives
\[|F_p(e^{-\frac{a}{2}\phi(r)} v_j) - F_p(e^{-\frac{a}{2}\phi(r)} v_{j-1})|
  \lesssim \la r\ra^{\frac{-ap}{2}}
  \Bigl(|v_j|^{p-1}+|v_{j-1}|^{p-1}\Bigr) |v_j-v_{j-1}|,\]
by \eqref{rpz0}, we get
\begin{multline*}
  \|v_{j+1}-v_j\|_X\lesssim \|\la r\ra^{1-s-\frac{(p-1)a}{2}} \la
  t-r\ra^{\frac{1}{2}}
  \Bigl(|v_j|^{p-1}+|v_{j-1}|^{p-1}\Bigr)|v_j-v_{j-1}|\|_{\ell^1_U
    \ell^2_{\tau\ge 4U} L^2_{t,x}(C^U_\tau)}
  \\+  \|\la r\ra^{\frac{3}{2}-s-\frac{(p-1)a}{2}} 
  \Bigl(|v_j|^{p-1}+|v_{j-1}|^{p-1}\Bigr)|v_j-v_{j-1}|\|_{\ell^2_\tau
    \ell^2_{R\le \tau/2} L^2_{t,x}(C^R_\tau)}.
\end{multline*}
By \eqref{Z2U} and \eqref{Z2R}, this is in turn controlled by
\[\Bigl(\|Z^{\le 2} v_j\|_X^{p-1} + \|Z^{\le
    2}v_{j-1}\|^{p-1}_X\Bigr)\|v_j-v_{j-1}\|_X.\]
Thus, if we apply \eqref{bdd_goal} with $\varepsilon$ sufficiently
small, we obtain \eqref{Cauchy}, which completes the proof.

\vskip .25in

\begin{thebibliography}{10}

\bibitem{AMW}
Elliot Adamson, Jason Metcalfe, and Chengbo Wang.
\newblock The {S}trauss conjecture for the wave equation with scattering
  damping.
\newblock {P}reprint, 2026.

\bibitem{Alinhac_ghost}
S.~Alinhac.
\newblock The null condition for quasilinear wave equations in two space
  dimensions {I}.
\newblock {\em Invent. Math.}, 145(3):597--618, 2001.

\bibitem{Dabbicco-Lucente-Reissig}
Marcello D'Abbicco, Sandra Lucente, and Michael Reissig.
\newblock A shift in the {S}trauss exponent for semilinear wave equations with
  a not effective damping.
\newblock {\em J. Differential Equations}, 259(10):5040--5073, 2015.

\bibitem{DaRo_rp}
Mihalis Dafermos and Igor Rodnianski.
\newblock A new physical-space approach to decay for the wave equation with
  applications to black hole spacetimes.
\newblock In {\em X{VI}th {I}nternational {C}ongress on {M}athematical
  {P}hysics}, pages 421--432. World Sci. Publ., Hackensack, NJ, 2010.

\bibitem{FaWa}
Daoyuan Fang and Chengbo Wang.
\newblock Weighted {S}trichartz estimates with angular regularity and their
  applications.
\newblock {\em Forum Math.}, 23(1):181--205, 2011.

\bibitem{GLS97}
Vladimir Georgiev, Hans Lindblad, and Christopher~D. Sogge.
\newblock Weighted {S}trichartz estimates and global existence for semilinear
  wave equations.
\newblock {\em Amer. J. Math.}, 119(6):1291--1319, 1997.

\bibitem{He-Lai}
Daoyin He and Ning-An Lai.
\newblock Morawetz type estimate for damped wave equation in high dimension and
  its application.
\newblock {\em Nonlinearity}, 39(8):Paper No. 085020, 26, 2026.

\bibitem{HMSSZ}
Kunio Hidano, Jason Metcalfe, Hart~F. Smith, Christopher~D. Sogge, and Yi~Zhou.
\newblock On abstract {S}trichartz estimates and the {S}trauss conjecture for
  nontrapping obstacles.
\newblock {\em Trans. Amer. Math. Soc.}, 362(5):2789--2809, 2010.

\bibitem{Ikehata-Todorova-Yordanov}
Ryo Ikehata, Grozdena Todorova, and Borislav Yordanov.
\newblock Critical exponent for semilinear wave equations with space-dependent
  potential.
\newblock {\em Funkcial. Ekvac.}, 52(3):411--435, 2009.

\bibitem{John79}
Fritz John.
\newblock Blow-up of solutions of nonlinear wave equations in three space
  dimensions.
\newblock {\em Manuscripta Math.}, 28(1-3):235--268, 1979.

\bibitem{Lai-Liu-Tu-Wang}
Ning-An Lai, Mengyun Liu, Ziheng Tu, and Chengbo Wang.
\newblock Lifespan estimates for semilinear wave equations with space dependent
  damping and potential.
\newblock {\em Calc. Var. Partial Differential Equations}, 62(2):Paper No. 44,
  32, 2023.

\bibitem{Lai-Takamura}
Ning-An Lai and Hiroyuki Takamura.
\newblock Blow-up for semilinear damped wave equations with subcritical
  exponent in the scattering case.
\newblock {\em Nonlinear Anal.}, 168:222--237, 2018.

\bibitem{Lai-Tu}
Ning-An Lai and Ziheng Tu.
\newblock Strauss exponent for semilinear wave equations with scattering space
  dependent damping.
\newblock {\em Journal of Mathematical Analysis and Applications}, page 124189,
  2020.

\bibitem{Lai-Zhou}
Ning-An Lai and Yi~Zhou.
\newblock Global existence for semilinear wave equations with scaling invariant
  damping in 3-{D}.
\newblock {\em Nonlinear Anal.}, 210:Paper No. 112392, 12, 2021.

\bibitem{Li}
Xinfu Li.
\newblock Critical exponent for semilinear wave equation with critical
  potential.
\newblock {\em NoDEA Nonlinear Differential Equations Appl.}, 20(3):1379--1391,
  2013.

\bibitem{LMSTW}
Hans Lindblad, Jason Metcalfe, Christopher~D. Sogge, Mihai Tohaneanu, and
  Chengbo Wang.
\newblock The {S}trauss conjecture on {K}err black hole backgrounds.
\newblock {\em Math. Ann.}, 359(3-4):637--661, 2014.

\bibitem{Liu-Wang}
Mengyun Liu and Chengbo Wang.
\newblock Global existence for semilinear damped wave equations in relation
  with the {S}trauss conjecture.
\newblock {\em Discrete Contin. Dyn. Syst.}, 40(2):709--724, 2020.

\bibitem{Looi-Tohaneanu}
Shi-Zhuo Looi and Mihai Tohaneanu.
\newblock Global existence and pointwise decay for nonlinear waves under the
  null condition.
\newblock {\em Dyn. Partial Differ. Equ.}, 23(3):225--268, 2026.

\bibitem{M-Rhoads}
Jason Metcalfe and Taylor Rhoads.
\newblock Long-time existence for systems of quasilinear wave equations.
\newblock {\em Matematica}, 2(1):37--84, 2023.

\bibitem{MetSo06}
Jason Metcalfe and Christopher~D. Sogge.
\newblock Long-time existence of quasilinear wave equations exterior to
  star-shaped obstacles via energy methods.
\newblock {\em SIAM J. Math. Anal.}, 38(1):188--209, 2006.

\bibitem{M-Stewart}
Jason Metcalfe and Alexander Stewart.
\newblock On a system of weakly null semilinear wave equations.
\newblock {\em Anal. Math. Phys.}, 12(5):Paper No. 125, 23, 2022.

\bibitem{MTT}
Jason Metcalfe, Daniel Tataru, and Mihai Tohaneanu.
\newblock Price's law on nonstationary space-times.
\newblock {\em Adv. Math.}, 230(3):995--1028, 2012.

\bibitem{MW}
Jason Metcalfe and Chengbo Wang.
\newblock The {S}trauss conjecture on asymptotically flat space-times.
\newblock {\em SIAM J. Math. Anal.}, 49(6):4579--4594, 2017.

\bibitem{Nishihara-Sobajima-Wakasugi}
Kenji Nishihara, Motohiro Sobajima, and Yuta Wakasugi.
\newblock Critical exponent for the semilinear wave equations with a damping
  increasing in the far field.
\newblock {\em NoDEA Nonlinear Differential Equations Appl.}, 25(6):Paper No.
  55, 32, 2018.

\bibitem{Li-Yin}
Li~Qianqian and Yin Huicheng.
\newblock Global small data weak solutions of 2-{D} semilinear wave equations
  with scale-invariant damping, {III}.
\newblock {\em J. Differential Equations}, 455:Paper No. 113972, 34, 2026.

\bibitem{Sideris}
Thomas~C. Sideris.
\newblock Nonexistence of global solutions to semilinear wave equations in high
  dimensions.
\newblock {\em J. Differential Equations}, 52(3):378--406, 1984.

\bibitem{Sobajima}
Motohiro Sobajima.
\newblock On global existence for semilinear wave equations with
  space-dependent critical damping.
\newblock {\em J. Math. Soc. Japan}, 75(2):603--627, 2023.

\bibitem{Sterbenz}
Jacob Sterbenz.
\newblock Angular regularity and {S}trichartz estimates for the wave equation.
\newblock {\em Int. Math. Res. Not.}, (4):187--231, 2005.
\newblock With an appendix by Igor Rodnianski.

\bibitem{Ta01-2}
Daniel Tataru.
\newblock Strichartz estimates in the hyperbolic space and global existence for
  the semilinear wave equation.
\newblock {\em Trans. Amer. Math. Soc.}, 353(2):795--807 (electronic), 2001.

\end{thebibliography}

\end{document}